\documentclass[11pt]{article}

\usepackage[margin=1.15in]{geometry}
\usepackage{amsmath,amssymb,amsthm}
\usepackage{hyperref}

\newtheorem{theorem}{Theorem}[section]
\newtheorem{proposition}[theorem]{Proposition}
\newtheorem{lemma}[theorem]{Lemma}
\newtheorem{remark}[theorem]{Remark}

\title{One-point extensions of Euclidean Ramsey sets}
\author{Mostafa Mirabi}
\date{}

\begin{document}

\maketitle


\begin{abstract}
Let $X$ be a finite Euclidean Ramsey set. We prove that adjoining any point outside the affine hull of $X$ gives another Euclidean Ramsey set, answering a conjecture of Ivan, Leader, and Walters. We first give an elementary product proof under the additional assumptions that the orthogonal projection of the new point lies in $\operatorname{conv}(X)$ and that its distance from $\operatorname{aff}(X)$ is sufficiently large. We then prove the general case by combining the product theorem for $E$-Ramsey configurations and K\v{r}\'{\i}\v{z}'s orbit-gluing theorem with a cyclic construction. No transitivity assumption on $X$ is needed.
\end{abstract}


\section{Introduction}
A finite set $X$ in a Euclidean space is called Ramsey if, for every positive integer $k$, there is an integer $N$ such that every $k$-colouring of $\mathbb{R}^N$ contains a monochromatic isometric copy of $X$. The problem of determining which finite Euclidean sets are Ramsey goes back to Erd\H{o}s, Graham, Montgomery, Rothschild, Spencer and Straus \cite{EGMRSS}.

Ivan, Leader and Walters \cite{ILW} recently studied generalized prisms and obtained a number of new constructions of Ramsey sets. One consequence of their main result is that if $X$ is subsoluble, then adding a point outside the hyperplane containing $X$ again gives a Ramsey set. They asked whether subsolubility can be replaced simply by the assumption that $X$ is Ramsey. This is their Conjecture 8.

We prove the conjecture.

\begin{theorem}\label{thm:main}
Let $X$ be a finite Ramsey set in a Euclidean space, and let $z$ be a point outside $\operatorname{aff}(X)$. Then $X\cup\{z\}$ is Ramsey.
\end{theorem}

The argument developed in two steps. The first is an elementary diagonal product construction. If the orthogonal projection of the new point lies in $\operatorname{conv}(X)$, it proves the result once the height is sufficiently large. This is Proposition~\ref{prop:large-height}. The same argument also shows why the height cannot in general be made arbitrarily small by that particular construction.

The second step is the one that removes the restriction. We use K\v{r}\'i\v{z}'s language of $E$-Ramsey configurations \cite{Kriz}. A cyclic product construction gives one-point extensions at heights
$$
\frac{\lVert y-x_0\rVert}{\sqrt{n}},
$$
which tend to zero, and K\v{r}\'i\v{z}'s orbit-gluing theorem makes the two required classes monochromatic. The usual product theorem then increases the height to any prescribed nonzero value.

\textbf{Note.} The argument here was obtained independently of the recent preprint of Moore \cite{Moore}. The partial result in Proposition~\ref{prop:large-height} and  the proof of Theorem~\ref{thm:main} were circulated privately to Ivan, Leader and Walters in June 2026.  Moore's preprint appeared on arXiv on August 10, 2026. His proof is different from the one given here.


\section{A first product argument}
We use the standard facts that the Ramsey property is invariant under nonzero scaling, is inherited by subsets, and is closed under finite Cartesian products \cite{EGMRSS}.

Let
$
X=\{x_1,\ldots,x_m\}\subseteq \mathbb{R}^d
$ 
be Ramsey, and let $y\in\operatorname{conv}(X)$. Define
$$
\rho_X(y)^2=
\min\left\{
\sum_{i=1}^m p_i\lVert x_i-y\rVert^2:
 p_i\geq 0,\ \sum_{i=1}^m p_i=1,\ \sum_{i=1}^m p_i x_i=y
\right\}.
$$
The admissible set is nonempty and compact, so the minimum exists.

\begin{proposition}\label{prop:large-height}
Let $X=\{x_1,\ldots,x_m\}\subseteq\mathbb{R}^d$ be a finite Ramsey set, and let $y\in\operatorname{conv}(X)$. If $\lambda\neq 0$ and
$
|\lambda|\geq \rho_X(y),
$
then
$
(X\times\{0\})\cup\{(y,\lambda)\}\subseteq\mathbb{R}^{d+1}
$
is Ramsey.
\end{proposition}

\begin{proof}
By reflecting in the last coordinate, it is enough to consider $\lambda>0$. Choose admissible weights $p_1,\ldots,p_m$ such that
$$
\sigma^2=\sum_{i=1}^m p_i\lVert x_i-y\rVert^2\leq \lambda^2.
$$
For each $i$, the set $\sqrt{p_i}X$ is Ramsey when $p_i>0$, while for $p_i=0$ it is the one-point set $\{0\}$. Hence
$
Y=\prod_{i=1}^m \sqrt{p_i}X
$
is Ramsey.

Inside $Y$, define
$$
D(x)=(\sqrt{p_1}x,\ldots,\sqrt{p_m}x),\qquad x\in X,
$$
and
$
w=(\sqrt{p_1}x_1,\ldots,\sqrt{p_m}x_m).
$ 
For $x,x'\in X$,
$$
\lVert D(x)-D(x')\rVert^2
=\sum_{i=1}^m p_i\lVert x-x'\rVert^2
=\lVert x-x'\rVert^2.
$$
Thus $D(X)$ is an isometric copy of $X$.

Also, for $x\in X$,
\begin{align*}
\lVert D(x)-w\rVert^2
&=\sum_{i=1}^m p_i\lVert x-x_i\rVert^2\\
&=\sum_{i=1}^m p_i\lVert (x-y)+(y-x_i)\rVert^2\\
&=\lVert x-y\rVert^2
+2\left\langle x-y,\sum_{i=1}^m p_i(y-x_i)\right\rangle
+\sum_{i=1}^m p_i\lVert x_i-y\rVert^2\\
&=\lVert x-y\rVert^2+\sigma^2,
\end{align*}
since $\sum_i p_i x_i=y$. Therefore
$
D(X)\cup\{w\}
$
is isometric to
$
(X\times\{0\})\cup\{(y,\sigma)\}.
$
As a subset of the Ramsey set $Y$, it is Ramsey.

If $\lambda=\sigma$, we are done. If $\lambda>\sigma$, put
$
a=\sqrt{\lambda^2-\sigma^2}.
$
The two-point set $\{0,a\}$ is Ramsey, so $Y\times\{0,a\}$ is Ramsey. It contains
$
\{(D(x),0):x\in X\}\cup\{(w,a)\},
$
and for $x\in X$,
$$
\lVert (D(x),0)-(w,a)\rVert^2
=\lVert x-y\rVert^2+\sigma^2+a^2
=\lVert x-y\rVert^2+\lambda^2.
$$
Hence this subset is isometric to
$
(X\times\{0\})\cup\{(y,\lambda)\}.
$
The case $\lambda<0$ follows by reflection.
\end{proof}

\begin{remark}\label{rem:obstruction}
The proposition does not by itself give arbitrary height. If $y\in\operatorname{conv}(X)\setminus X$, then $\rho_X(y)>0$. In fact, in any diagonal construction of the form used above, the barycentre condition forces the extra squared height to be
$
\sum_i p_i\lVert x_i-y\rVert^2.
$ 
Thus a different idea is needed to produce heights tending to zero.
\end{remark}



\section{\texorpdfstring{$E$-Ramsey}{E-Ramsey} configurations and the full theorem}

We recall the part of K\v{r}\'i\v{z}'s machinery that we need. A finite subset of a Euclidean space will be called a configuration. Let $F$ be a configuration and let $E$ be an equivalence relation on $F$. We say that $F$ is $E$-Ramsey if, for every positive integer $k$, there is an integer $N$ such that every $k$-colouring of $\mathbb{R}^N$ admits an isometric embedding $\varphi:F\to\mathbb{R}^N$ with
$$
xEy\quad\Longrightarrow\quad
\operatorname{col}(\varphi(x))=\operatorname{col}(\varphi(y)).
$$
Ordinary Ramsey-ness is the special case in which $E$ has one equivalence class.

We use two results of K\v{r}\'i\v{z} \cite[Theorems 3.2 and 4.1]{Kriz}.

\begin{itemize}
\item If $F_1$ is $E_1$-Ramsey and $F_2$ is $E_2$-Ramsey, then $F_1\times F_2$ is $(E_1\times E_2)$-Ramsey. The same holds for finite products.

\item Suppose $F$ is $E$-Ramsey and $b:F\to F$ is an isometry respecting $E$. For $z\in F$ and $r\geq 1$, let $U(E;z,b,r)$ be the smallest equivalence relation containing $E$ in which
$
z,bz,\ldots,b^{r-1}z
$
are equivalent. Then $F$ is $U(E;z,b,r)$-Ramsey.
\end{itemize}

The next simple observation lets us add auxiliary points without imposing any colour condition on them.
\begin{lemma}\label{lem:free-extension}
Let $X\subseteq\mathbb{R}^d$ be a nonempty finite Ramsey set, and let $C\subseteq\mathbb{R}^d$ be finite with $X\subseteq C$. Let $E_C$ be the equivalence relation on $C$ whose equivalence classes are $X$ and the singleton sets $\{c\}$ for $c\in C\setminus X$. Then $C$ is $E_C$-Ramsey.
\end{lemma}

\begin{proof}
Fix $k$. Choose $N_0$ such that every $k$-colouring of $\mathbb{R}^{N_0}$ contains a monochromatic copy of $X$. We claim that $\mathbb{R}^{N_0+d}$ witnesses that $C$ is $E_C$-Ramsey.

Take a $k$-colouring of $\mathbb{R}^{N_0+d}$ and restrict it to the coordinate subspace $\mathbb{R}^{N_0}\times\{0\}$. There is a monochromatic isometric copy $\varphi(X)$ in this subspace. Fix $x_*\in X$, and put
$
V=\operatorname{span}(X-x_*).
$
Let
$
V'=\operatorname{span}(\varphi(X)-\varphi(x_*)).
$
The map $\varphi$ extends to an affine isometry
$$
T:x_*+V\longrightarrow \varphi(x_*)+V'.
$$
Every $c\in C$ can be written uniquely as
$
c=x_*+p(c)+q(c),
$
with $p(c)\in V$ and $q(c)\in V^\perp$.

Since $\dim V' = \dim V$, the space $(V')^\perp\subseteq\mathbb{R}^{N_0+d}$ has dimension at least $d-\dim V$. We may therefore choose a linear isometric embedding
$$
J:V^\perp\longrightarrow (V')^\perp.
$$
Define
$$
\Phi(c)=T(x_*+p(c))+J(q(c)).
$$
Differences coming from the two terms are orthogonal, so $\Phi$ is an isometric embedding of $C$. For $x\in X$ we have $q(x)=0$, hence $\Phi(x)=\varphi(x)$. Thus the whole $E_C$-class $X$ is monochromatic, while the remaining classes are singletons. This proves the lemma.
\end{proof}

We now prove the coordinate form of Theorem~\ref{thm:main}.

\begin{theorem}\label{thm:coordinate}
Let $X\subseteq\mathbb{R}^d$ be a nonempty finite Ramsey set. Let $y\in\mathbb{R}^d$ and let $\lambda\in\mathbb{R}\setminus\{0\}$. Then
$$
Z=(X\times\{0\})\cup\{(y,\lambda)\}\subseteq\mathbb{R}^{d+1}
$$
is Ramsey.
\end{theorem}

\begin{proof}
By reflection in the last coordinate, it is enough to treat $\lambda>0$.

First suppose that $y\in X$. Then $Z$ is a subset of
$
X\times\{0,\lambda\}.
$
The two-point set $\{0,\lambda\}$ is Ramsey, so the product theorem shows that $X\times\{0,\lambda\}$ is Ramsey. Hence $Z$ is Ramsey.

Now assume $y\notin X$, and fix $x_0\in X$. For an integer $n\geq 1$, define
$$
a_i=\left(1-\frac{i}{n}\right)x_0+\frac{i}{n}y,
\qquad i=0,1,\ldots,n.
$$
Thus $a_0=x_0$ and $a_n=y$. Put
$
C_n=X\cup\{a_0,a_1,\ldots,a_n\},
$
and let $E_n$ be the equivalence relation on $C_n$ whose only non-singleton class is $X$. By Lemma~\ref{lem:free-extension}, $C_n$ is $E_n$-Ramsey. The product theorem therefore gives that
$
C_n^{n+1}
$
is $E_n^{n+1}$-Ramsey, where the equivalence relation is taken coordinatewise.

Inside $C_n^{n+1}$, define
$
D(x)=(x,a_1,a_2,\ldots,a_n),$ for $ x\in X.
$
Let $b$ be the cyclic coordinate shift
$
b(u_0,u_1,\ldots,u_n)=(u_n,u_0,u_1,\ldots,u_{n-1}),
$
and set
$$
F_n=\bigcup_{r=0}^n b^rD(X).
$$
Let $E$ be the restriction of $E_n^{n+1}$ to $F_n$. Since the $E$-Ramsey property is inherited by subsets, $F_n$ is $E$-Ramsey. The map $b$ restricts to an isometry of $F_n$ and respects $E$.

Set
$
z_0=D(x_0)=(x_0,a_1,\ldots,a_n).
$
By K\v{r}\'i\v{z}'s orbit-gluing theorem with $r=2$, the configuration $F_n$ is $U(E;z_0,b,2)$-Ramsey. The $E$-class of $z_0$ contains
$$
D(X)=\{(x,a_1,\ldots,a_n):x\in X\},
$$
and
$$
bz_0=(a_n,a_0,a_1,\ldots,a_{n-1})
=(y,x_0,a_1,\ldots,a_{n-1}).
$$
The relation $U(E;z_0,b,2)$ merges the $E$-classes of $z_0$ and $bz_0$. It follows that
$
S_n=D(X)\cup\{bz_0\}
$ 
is Ramsey.

It remains only to compute the geometry of $S_n$. For $x,x'\in X$,
$
\lVert D(x)-D(x')\rVert^2=\lVert x-x'\rVert^2.
$ 
Also, for $x\in X$,
\begin{align*}
\lVert D(x)-bz_0\rVert^2
&=\lVert x-y\rVert^2+
\sum_{i=1}^n \lVert a_i-a_{i-1}\rVert^2.
\end{align*}
But
$
a_i-a_{i-1}=\frac{y-x_0}{n},
$ 
so
$$
\sum_{i=1}^n \lVert a_i-a_{i-1}\rVert^2
=\frac{\lVert y-x_0\rVert^2}{n}.
$$
Therefore $S_n$ is isometric to
$$
(X\times\{0\})\cup
\left\{\left(y,\frac{\lVert y-x_0\rVert}{\sqrt{n}}\right)\right\}.
$$
Write
$$
h_n=\frac{\lVert y-x_0\rVert}{\sqrt{n}}.
$$
Then $h_n\to 0$. Choose $n$ so large that $h_n\leq\lambda$. If $h_n=\lambda$, we are done. Otherwise let
$
t=\sqrt{\lambda^2-h_n^2}.
$ 
The product of the Ramsey set
$
(X\times\{0\})\cup\{(y,h_n)\}
$
with the Ramsey two-point set $\{0,t\}$ is Ramsey. It contains
$
\{(x,0,0):x\in X\}\cup\{(y,h_n,t)\},
$ 
which is isometric to
$
(X\times\{0\})\cup\{(y,\lambda)\}
$
because $h_n^2+t^2=\lambda^2$. This finishes the proof.
\end{proof}

\begin{proof}[Proof of Theorem~\ref{thm:main}]
Work in the affine span of $X\cup\{z\}$. After an isometry, we may identify $\operatorname{aff}(X)$ with $\mathbb{R}^d\times\{0\}$ and write
$
z=(y,\lambda)
$
with $y\in\mathbb{R}^d$ and $\lambda\neq 0$. The result is then exactly Theorem~\ref{thm:coordinate}.
\end{proof}

\begin{remark}
The role of K\v{r}\'i\v{z}'s theorem is quite specific. The cyclic construction by itself produces several copies of $X$ arranged around a product, but there is no transitive group available because $X$ is assumed only to be Ramsey. The $E$-Ramsey formulation lets us keep $X$ as one colour class, leave all auxiliary points unconstrained, and then glue two consecutive classes along the cyclic shift. This replaces the transitivity assumption used in the generalized-prism construction of \cite{ILW}.
\end{remark}

\section*{Acknowledgements}
The author would like to thank Maria-Romina Ivan, Imre Leader, and Mark Walters for formulating Conjecture 8. He is especially grateful to Imre Leader and Maria-Romina Ivan for helpful discussions concerning the conjecture and for reading earlier versions of the argument.

\medskip
\noindent 
The Taft School, Watertown CT 06795, USA, and\newline Wesleyan University, Middletown CT 06459, USA.\\
\noindent 
Email: \hyperlink{}{mmirabi@wesleyan.edu}\\
Website: \hyperlink{}{https://sites.google.com/site/mostafamirabi/}

\end{document}